\documentclass[12pt]{amsart}
\usepackage{amscd,amsmath,amsthm,amssymb,graphics}
\usepackage{pstcol,pst-plot,pst-3d}%\usepackage[T1]{fontenc}
\usepackage{lmodern,pst-node}%\usepackage{geometric}
\usepackage{multicol}
\usepackage{epic,eepic}
\usepackage{amsfonts,amssymb,amscd,amsmath,enumerate,verbatim}
\psset{unit=0.7cm,linewidth=0.8pt,arrowsize=2.5pt 4}
\newpsstyle{fatline}{linewidth=1.5pt}
\newpsstyle{fyp}{fillstyle=solid,fillcolor=verylight}
\definecolor{verylight}{gray}{0.97}
\definecolor{light}{gray}{0.9}
\definecolor{medium}{gray}{0.85}
\definecolor{dark}{gray}{0.6}

\def\frk{\frak}               % font for "Fraktur"

\newcommand{\fp}{\frak{p}}

\def\mm{{\frk m}}

\def\Phi{{\frk n}}
\def\Phi{{\frk N}}
\def\opn#1#2{\def#1{\operatorname{#2}}} % to make operators
\opn\chara{char} \opn\length{\ell} \opn\pd{pd} \opn\rk{rk}
\opn\projdim{proj\,dim} \opn\injdim{inj\,dim} \opn\rank{rank}
\opn\depth{depth} \opn\grade{grade} \opn\height{height}
\opn\embdim{emb\,dim} \opn\codim{codim}

\opn\Tr{Tr} \opn\bigrank{big\,rank}
\opn\superheight{superheight}\opn\lcm{lcm}
\opn\trdeg{tr\,deg}%\emph{
\opn\reg{reg} \opn\lreg{lreg} \opn\ini{in} \opn\lpd{lpd}
\opn\size{size}\opn\bigsize{bigsize}
\opn\cosize{cosize}\opn\bigcosize{bigcosize}
\opn\sdepth{sdepth}\opn\sreg{sreg}
\opn\link{link}\opn\fdepth{fdepth}
\opn\div{div} \opn\Div{Div} \opn\cl{cl} \opn\Cl{Cl}
\opn\Spec{Spec} \opn\Supp{Supp} \opn\supp{supp} \opn\Sing{Sing}
\opn\Ass{Ass} \opn\Min{Min}\opn\Mon{Mon} \opn\dstab{dstab} \opn\astab{astab}
\opn\Syz{Syz}
\opn\Ann{Ann} \opn\Rad{Rad} \opn\Soc{Soc}
\opn\Im{Im} \opn\Ker{Ker} \opn\Coker{Coker} \opn\Am{Am}
\opn\Hom{Hom} \opn\Tor{Tor} \opn\Ext{Ext} \opn\End{End}
\opn\Aut{Aut} \opn\id{id}

\opn\nat{nat}
\opn\pff{pf}%   \pf exists already
\opn\Pf{Pf} \opn\GL{GL} \opn\SL{SL} \opn\mod{mod} \opn\ord{ord}
\opn\Gin{Gin} \opn\Hilb{Hilb}\opn\sort{sort}
\opn\initial{init}
\opn\ende{end}
\opn\height{height}
\opn\type{type}
\opn\aff{aff} \opn\con{conv} \opn\relint{relint} \opn\st{st}
\opn\lk{lk} \opn\cn{cn} \opn\core{core} \opn\vol{vol}
\opn\link{link} \opn\star{star}\opn\lex{lex}
\opn\gr{gr}

\def\pot#1#2{#1[\kern-0.28ex[#2]\kern-0.28ex]}

\opn\dirlim{\underrightarrow{\lim}}
\opn\inivlim{\underleftarrow{\lim}}
\let\to=\rightarrow

\def\Implies{\ifmmode\Longrightarrow \else
        \unskip${}\Longrightarrow{}$\ignorespaces\fi}
\def\implies{\ifmmode\Rightarrow \else
        \unskip${}\Rightarrow{}$\ignorespaces\fi}
\def\iff{\ifmmode\Longleftrightarrow \else
        \unskip${}\Longleftrightarrow{}$\ignorespaces\fi}

\let\:=\colon
\newtheorem{Theorem}{Theorem}[section]
 \newtheorem{Lemma}[Theorem]{Lemma}
 \newtheorem{Corollary}[Theorem]{Corollary}
 \newtheorem{Proposition}[Theorem]{Proposition}
 \newtheorem{Remark}[Theorem]{Remark}
 
 \newtheorem{Example}[Theorem]{Example}
 
 \newtheorem{Definition}[Theorem]{Definition}

\let\epsilon\varepsilon
\let\kappa=\varkappa
\def\qed{\ifhmode\textqed\fi
      \ifmmode\ifinner\quad\qedsymbol\else\dispqed\fi\fi}
\def\textqed{\unskip\nobreak\penalty50
       \hskip2em\hbox{}\nobreak\hfil\qedsymbol
       \parfillskip=0pt \finalhyphendemerits=0}
\def\dispqed{\rlap{\qquad\qedsymbol}}

\opn\dis{dis}
\def\pnt{{\raise0.5mm\hbox{\large\bf.}}}

\opn\Lex{Lex}

\begin{document}
% \dedicatory{Department of Mathematics, University of Kurdistan, P.O. Box: 416, Sanandaj,
%Iran. Shkarimi2011@yahoo.com and a\_mafi@ipm.ir.}
 \title{On stability properties of powers of polymatroidal ideals}

 \author {Shokoufe Karimi and Amir Mafi*}

\address{Sh. Karimi, Department of Mathematics, University of Kurdistan, P.O. Box: 416, Sanandaj,
Iran.}
\email{Shkarimi2011@yahoo.com}

\address{A. Mafi, Department of Mathematics, University of Kurdistan, P.O. Box: 416, Sanandaj,
Iran.}
\email{a\_mafi@ipm.ir}


\keywords{Associated primes, polymatroidal ideal, depth stability number.\\
* Corresponding author}

\begin{abstract}
Let $R=K[x_1,...,x_n]$ be the polynomial ring in $n$ variables over a field $K$ with the maximal ideal $\frak{m}=(x_1,...,x_n)$.
 Let $\astab(I)$ and $\dstab(I)$ be the smallest integer $n$ for which $\Ass(I^n)$ and $\depth(I^n)$ stabilize, respectively.
  In this paper we show that $\astab(I)=\dstab(I)$ in the following cases:
  \begin{itemize}
\item[(i)] $I$ is a matroidal ideal and $n\leq 5$.
\item[(ii)]  $I$ is a polymatroidal ideal, $n=4$ and $\frak{m}\notin\Ass^{\infty}(I)$, where $\Ass^{\infty}(I)$ is the stable set of associated
prime ideals of $I$.
\item[(iii)] $I$ is a polymatroidal ideal of degree $2$.
\end{itemize}
Moreover, we give an example of a polymatroidal ideal for which $\astab(I)\neq\dstab(I)$. This is a counterexample to the conjecture of Herzog and Qureshi,  according to which these two numbers are the same for polymatroidal ideals.
\end{abstract}

\maketitle

\section*{Introduction}
Throughout this paper, we assume that $R=K[x_1,...,x_n]$ is the polynomial ring in $n$ variables over a field $K$ with the maximal ideal $\frak{m}=(x_1,...,x_n)$, $I$ a monomial ideal of $R$ and $G(I)$ the unique minimal monomial generators set of $I$. It is customary to denote by $\Ass(I)$ the set of associated prime ideals of $R/I$. Brodmann \cite{B1} showed that there exists an integer $k_0$ such that $\Ass(I^k)=\Ass(I^{k_0})$ for all $k\geq k_0$. The smallest such integer $k_0$ is
called the {\it index of Ass-stability} of $I$,  and denoted by $\astab(I)$. Moreover,   $\Ass(I^{k_0})$ is called the  stable set of associated
prime ideals of $I$. It is denoted by $\Ass^{\infty}(I)$. Brodmann \cite{B} also showed that there exists an integer $k_0$ such that $\depth R/I^k=\depth R/I^{k_0}$ for all $k\geq k_0$. The smallest such integer $k_0$ is called the {\it index of depth stability} of $I$ and denoted by $\dstab(I)$. This
stable depth is called the limit depth of $I$, and is denoted  by $\lim_{k\to\infty}\depth R/I^k$. Moreover, he proved that $$\lim_{k\to\infty}\depth R/I^k\leq n-\ell(I), (\dag)$$ where $\ell(I)$ is the analytic spread of $I$, that is, the dimension of $\mathcal{R}(I)/{{\frak{m}}\mathcal{R}(I)}$. Here $\mathcal{R}(I)=\oplus_{k\geq 0}I^kt^k$ is the Rees ring of $I$. Eisenbud and Huneke \cite{EH} showed that the equality $(\dag)$ holds, if the associated graded ring $gr_I(R)$ is Cohen-Macaulay.

Herzog and Mafi \cite{HM} proved that if $n=3$ then, for any graded ideal $I$ of $R$, $\astab(I)=\dstab(I)$. Also, they showed that for $n=4$ the indices $\astab(I)$ and $\dstab(I)$ are unrelated.
Herzog, Rauf and Vladoiu \cite{HRV} showed that for every polymatroidal ideal of Veronese type $\astab(I)=\dstab(I)$ and for every transversal polymatroidal ideal
$\astab(I)=1=\dstab(I)$.
Herzog and Qureshi \cite{HQ} proved that if $I$ is a polymatroidal ideal of $R$, then $\astab(I),\dstab(I)<\ell(I)$ and they conjectured that $\astab(I)=\dstab(I)$ for all polymatroidal ideal $I$.
In this paper we study this conjecture. More precisely, our main results of this paper are as follows:
\begin{Theorem}
Let one of the following cases holds:
\begin{itemize}
\item[(i)] $I$ is a matroidal ideal and $n\leq 5$.
\item[(ii)]  $I$ is a polymatroidal ideal, $n=4$ and $\frak{m}\notin\Ass^{\infty}(I)$.
\item[(iii)] $I$ is a polymatroidal ideal of degree 2.
\item[(iv)] $I$ satisfying in the strong exchange property.
\end{itemize}
Then $\astab(I)=\dstab(I)$.
\end{Theorem}
In addition, we give a counterexample to {\it the conjecture of Herzog and Qureshi}.

For any unexplained notion or terminology, we refer the reader to \cite{HH1}, \cite{V}, \cite{HS} and \cite{BH1}.
Several explicit examples were  performed with help of the computer algebra systems CoCoA \cite{AB} and Macaulay2 \cite{GS},  as well as with the program in
\cite{BHR} which allows one to compute $\Ass^\infty(I)$ of a monomial ideal $I$.

\section{Preliminary}
In this section, we collect notations, terminology and basic results used in this paper.

Let, as before, $K$ be a field and $R=K[x_1,...,x_n]$ be the polynomial ring in $n$ variables over $K$ with each $\deg x_i=1$ and $\frak{m}=(x_1,...,x_n)$ the unique homogenous maximal ideal of $R$. For a monomial ideal $I$ of $R$ and $G(I)=\{u_1,...,u_t\}$, we set $\supp(I)=\cup_{i=1}^t\supp(u_i)$, where $\supp(u)=\{x_i: u=x_1^{a_1}...x_n^{a_n}, a_i\neq 0\}$ and we set $\gcd(I)=\gcd(u_1,...,u_m)$. The linear relation graph $\Gamma_I$ associated to a monomial ideal is the graph whose vertex set $V(\Gamma_I)$ is a subset of $\{x_1,...,x_n\}$ and for which $\{x_i,x_j\}\in E(\Gamma_I)$ if and only if there exist $u_k,u_l\in G(I)$ such that $x_iu_k=x_ju_l$ (see \cite[Definition 3.1]{HQ}). We say that the monomial ideal $I$ is full-supported if $\supp(I)=\{x_1,...,x_n\}$.

A monomial ideal $I$ is called a polymatroidal ideal, if it is generated in a single degree with the exchange property that for any two elements $u,v\in G(I)$ such that $\deg_{x_i}(u)>\deg_{x_i}(v)$ for some $i$, there exists an integer $j$ such that $\deg_{x_j}(u)<\deg_{x_j}(v)$ and $x_j(u/x_i)\in I$.
The polymatroidal ideal $I$ is called matroidal if $I$ is generated by square-free monomials. For a polymatroidal ideal $I$ one can compute the analytic spread as $\ell(I)=r-s+1$, where $r=|V(\Gamma_I)|$ and $s$ is the number of connected components of $\Gamma_I$ (see \cite[Lemma 4.2]{HQ}).

The product of polymatroidal ideals is again polymatroidal (see \cite[Theorem 5.3]{CH}). In particular each power of a polymatroidal ideal is polymatroidal. Also, $I$ is a polymatroidal ideal if and only if $(I:u)$ is a polymatroidal ideal for all monomial $u$ (see \cite[Theorem 1.1]{BH}).
According to \cite{HQ} and \cite{HRV}, every polymatroidal ideal satisfying the persistence property and non-increasing depth functions, that is, if $I$ is a polymatroidal ideal then, for all $k$, there is the following sequences: \\

$\Ass(I^k)\subseteq\Ass(I^{k+1})$
and
$\depth(R/I^{k+1})\leq\depth(R/I^k)$.\\

In addition, every polymatroidal ideal is a normal ideal (see \cite[Theorem 3.4]{HRV}) and consequently the Rees ring $\mathcal{R}(I)$ and also the associated graded ring $gr_I(R)$ is Cohen-Macaulay (see \cite[Theorem 7.2.35]{V}). In particular, the equality $(\dag)$ holds.

Herzog and Vladoiu \cite{HV} proved the following interesting results about matroidal ideals.
\begin{Theorem} Let $I$ be a matroidal ideal of $R$ generated in degree $d$, and denote as before by $s$ the number of connected components of $\Gamma_I$.
\begin{itemize}
\item[(i)] Then $s\leq d$. If in addition $I$ is full-supported and $\gcd(I)=1$, then $V(\Gamma_I)=\{x_1,...,x_n\}$ and $s=d$ if and only if $\dstab(I)=1$.
\item[(ii)] If $I$ is full-supported and $\gcd(I)=1$, then $I\subseteq {\frak{p}_1}\cap\cdots\cap{\frak{p}_s},$ where $\frak{p}_1,\ldots,{\frak{p}_s}$ are the monomial prime ideals generated by the sets of vertices of the connected components $ \Gamma_1,\ldots,\Gamma_s$ of $\Gamma_I$.
\item[(iii)] If $I$ is full-supported and $\gcd(I)=1$, then $\dstab(I)=1$ if and only if $I={\frak{p}_1}\cdots{\frak{p}_d}$, where ${\frak{p}_1},\cdots,{\frak{p}_d}$ are monomial prime ideals  in pairwise disjoint sets of variables.
\end{itemize}
\end{Theorem}
From Theorem 1.1 (iii) one can conclude that for all full-supported matroidal ideal with $\gcd(I)=1$ if $\dstab(I)=1$, then $\astab(I)=1$.

Herzog and Qureshi \cite{HQ} proved that if $I$ is a polymatroidal ideal of $R$, then $\astab(I),\dstab(I)<\ell(I)$ and it is well know that $\height(I)\leq \ell(I)$. One can ask whether $\astab(I),\dstab(I)\leq\height(I)$ for all polymatroidal ideals.
We give two examples to show that this does not hold.

\begin{Example}
Let $n=3$ and consider the polymatroid ideal $I=(x_1^3x_2x_3,x_1^2x_2^2x_3,x_1^3x_2^2)$. Then $\height(I)=1$ but $\astab(I)=\dstab(I)=2$.
\end{Example}

\begin{Example}
Let $n=4$ and consider the Veronese type ideal\\
$I=(x_1x_2x_3,x_1x_2x_4,x_1x_3x_4,x_2x_3x_4)$. Then $\gcd(I)=1$ and $\height(I)=2$ but $\astab(I)=\dstab(I)=3$.
\end{Example}

\section{The results}

We start this section by the following lemma.
\begin{Lemma}\label{L1}
Let $I\subset R$ be a full-supported polymatroidal ideal of degree $2$. Assume that $y_1$ and $y_2$ are variables in $R$ such that $u=y_1y_2\in G(I)$ and $y\notin\supp(u)$. Then $(u/y_1)y\in G(I)$ or $(u/y_2)y\in G(I)$.
\end{Lemma}

\begin{proof}
Since $I$ is full-supported, there exists an integer $1\leq i\leq n$ such that $v=yx_i\in G(I)$. If $x_i=y_1$ or $x_i=y_2$, then by the exchange property $(u/y_2)y\in G(I)$ or $(u/y_1)y\in G(I)$, respectively. If $x_i\neq y_1$, $y_2$, then by the exchange property $(u/y_2)y=(v/x_i)y_1\in G(I)$ or $(u/y_1)y=(v/x_i)y_2\in G(I)$. This complete the proof.
\end{proof}

In the following we recall the definition of linear quotients from \cite{HH}.
\begin{Definition}
We say that a monomial ideal $I\subset R$ has linear quotients if there is an ordering $u_1,...,u_t$ of the monomials belonging to $G(I)$ with $0<\deg u_1\leq\deg u_2\leq ...\leq\deg u_t$ such that, for each $2\leq j\leq t$, the colon ideal $(u_1,...,u_{j-1}):u_j$ is generated by a subset of $\{x_1,...,x_n\}$.
\end{Definition}

Let $I$ be a monomial ideal generated by a sequence $u_1,...,u_t$ with linear quotients. Let $q_j(I)$ denote the minimal number of linear forms generating $(u_1,...,u_{j-1}):u_j$, and $q(I)=\max\{q_j(I): 2\leq j\leq t\}$. Polymatroidal ideals have linear quotients with respect to the reverse lexicographical order of the generators, see \cite[Theorem 5.2]{CH}.

As in the proof of \cite[Corollary 1.6]{HT} the length of the minimal free resolution of $R/I$ over $R$ is equal to $q(I)+1$. Hence, by the Auslander-Buchsbaum formula, we have $\depth R/I=n-q(I)-1$. Chiang-Hsieh in \cite[Theorem 2.5]{C} proved that if $I\subset R$ is a full-supported matroidal ideal of degree $d$, then $\depth R/I=d-1$.

The following example says that the Chiang-Hsieh's theorem for polymatroidal ideals is not true .
\begin{Example}
Let $n=3$ and consider the polymatroidal ideal $I=(x_1,x_2,x_3)^2$. Then $d=2$ and $\depth R/I=0$.
\end{Example}
\begin{Lemma}\label{L2}
Let $I\subset R$ be a full-supported polymatroidal ideal in degree $2$. Then $q(I)\geq{n-2}$. In particular, if $m\notin\Ass(I)$ then $\depth R/I=1$.
\end{Lemma}

\begin{proof}
Since $I$ has linear quotients, there is an ordering $u_1,...,u_t$ of the monomials belonging to $G(I)$ such that, for each $2\leq j\leq t$, the colon ideal
$(u_1,...,u_{j-1}):u_j$ is generated by a subset of $\{x_1,...,x_n\}$. It is enough to show that $\{x_1,...,x_n\}\setminus\supp(u_t)\subseteq (u_1,...,u_{t-1}):u_t$.
By Lemma \ref{L1} if $y\notin\supp(u_t)$,  then $y\in (u_1,...,u_{t-1}):u_t$. Therefore $q(I)\geq{n-2}$, as required.
\end{proof}

In the sequel we recall the following definition and remark from \cite{T}.

\begin{Definition}
Let $I\subset R$ be a monomial ideal and let $x^t[i]$ denote the monomial $x_1^{t_1}...\widehat{x_i^{t_i}}...x_n^{t_n}$, where the term $x_i^{t_i}$ of $x^t$ is omitted.  For each $i=1,...,n$, we put $I[i]=(x^t[i]: x^t\in I)$.
\end{Definition}

\begin{Remark}\label{R1} Let $I$, $J$ be a monomial ideals of $R$. For $i=1,...,n$, we have
\begin{itemize}
\item[(a)] $I[i]=IR_{x_i}\cap R$, where $R_{x_i}$ is the localization of $R$ with respect to $x_i$.
\item[(b)] $I[i]=(I:x_i^{\infty})=\cup_{t=1}^{\infty}(I:x_i^t)$.
\item[(c)] $(I:J)[i]=(I[i]:J[i])$.
\item[(d)] $I^t[i]=I[i]^t$ for all $t\geq 1$.
\item[(e)] $\Ass(I^t)\setminus\{\mm\}=\cup_{i=1}^n\Ass(I[i]^t)$ for all $t\geq 1$.
\end{itemize}
\end{Remark}

\begin{Corollary}
Let $n=3$, $I\subset R$ be a monomial ideal and ${\mm}\notin\Ass^{\infty}(I)$. Then $\astab(I)=\dstab(I)=1$.
\end{Corollary}

\begin{proof}
By Remark \ref{R1}(e) and by using \cite[Remark 1.1]{HM}, for all $t\geq 1$, we have $\Ass(I^t)=\cup_{i=1}^3\Ass(I[i]^t)=\cup_{i=1}^3\Ass(I[i])=\Ass(I)$.
Thus $\astab(I)=1$ and so by \cite[Theorem 1.2]{HM} the result follows.
\end{proof}
\begin{Proposition}\label{P1}
Let $I\subset R$ be a polymatroidal ideal. Then $\astab(I[i])\leq\astab(I)$ for $i=1,...,n$.
\end{Proposition}

\begin{proof}
Set $\astab(I)=t$. It follows $\Ass(I^t)=\Ass(I^{t+1})=\cdots$. Since $I[i]$ is polymatroidal, we have $\Ass(I[i]^t)\subseteq\Ass(I[i]^{t+1})\subseteq \cdots$. Let $\frak{p}\in\Ass(I[i]^{t+1})$. From the exact sequence $$0\longrightarrow R/(I^{t+1}:x_i^{\infty})\overset{x_i^{\infty}}\longrightarrow R/I^{t+1}\longrightarrow R/(I^{t+1},x_i^{\infty})\longrightarrow 0,$$ we have $\Ass(I[i]^{t+1})\subseteq\Ass(I^{t+1})$. Thus $\frak{p}\in\Ass(I^{t+1})=\Ass(I^t)$ and so there exists a monomial $\alpha$ of $R$ such that $\frak{p}=(I^t:\alpha)$.
Since $x_i\notin\frak{p}$, we have $\frak{p}=\frak{p}[i]=(I^t:\alpha)[i]=(I^t[i]:\alpha[i])$. Therefore $\frak{p}\in\Ass(I[i]^t)$ and so $\Ass(I[i]^{t})=\Ass(I[i]^{t+1})$. Thus $\astab(I[i])\leq t$, as required.
\end{proof}

\begin{Proposition}\label{P2}
Let $I\subset R$ be a polymatroidal ideal. If ${\mm}\notin\Ass^{\infty}(I)$ or if ${\mm}\in\Ass(I)$, then $\astab(I)=\max\{\astab(I[i]): i=1,...,n\}$.
\end{Proposition}

\begin{proof}
 Set $t=\max\{\astab(I[i]): i=1,...,n\}$ and let ${\mm}\notin\Ass^{\infty}(I)$. Thus $\Ass(I[i]^t)=\Ass(I[i]^{t+1})$ for $i=1,...,n$. By Remark \ref{R1}(e), $\Ass(I^t)=\cup_{i=1}^n\Ass(I[i]^t)=\cup_{i=1}^n\Ass(I[i]^{t+1})$ and so $\Ass(I^t)=\Ass(I^{t+1})=\cdots.$ Therefore $\astab(I)\leq t$ and so by Proposition \ref{P1} the result follows.\\
If ${\mm}\in\Ass(I)$, then by using the same proof as above we have the result.
\end{proof}

The following example says that the condition ${\mm}\notin\Ass^{\infty}(I)$ or ${\mm}\in\Ass(I)$ in Proposition \ref{P2} is essential.
\begin{Example}
Let $n=3$ and $I=(x_1x_2,x_1x_3,x_2x_3)\subset R$. Then $\astab(I)=2$ and $\astab(I[i])=1$ for $i=1,2,3$.
\end{Example}

Let $I\subseteq R$ be a monomial ideal. In the case that $G(I)\subseteq S=K[x_{i_1},...,x_{i_k}]$ we denote by an abuse of notation the ideal $G(I)S$ again by $I$. Observe that by using this notion it follows that $\Ass_R(I)=\Ass_S(I)$ and so $\astab_R(I)=\astab_S(I)$. Also, it is well known that $\dstab_R(I)=\dstab_S(I)$. Thus, for computing $\astab(I)$ and $\dstab(I)$ we can always assume that $I$ is full-supported.

\begin{Lemma}\label{L3}
Let $I, J$ be polymatroidal ideals such that $\gcd (J)=1$. If $\alpha$ is a monomial element of $R$ such that $I={\alpha}J$, then $\astab(I)=\astab(J)$ and $\dstab(I)=\dstab(J)$.
\end{Lemma}

\begin{proof}
Since $I^t\cong J^t$ for all $t\geq 1$, it follows that $\pd(I^t)=\pd(J^t)$ and so by the Auslander-Buchsbaum formula $\depth(R/I^t)=\depth(R/J^t)$ for all $t\geq 1$. Thus $\dstab(I)=\dstab(J)$. By using \cite[Theorem 1.3]{KR}, for all $t\geq 1$, $\Ass({\alpha})\cup\Ass(J^t)\subseteq\Ass(I^t)$ and it is clear $\Ass(I^t)\subseteq\Ass({\alpha})\cup\Ass(J^t)$ for all $t\geq 1$. Therefore, for all $t\geq 1$,  $\Ass(I^t)=\Ass({\alpha})\cup\Ass(J^t)$ and so $\astab(I)=\astab(J)$, as required.
\end{proof}

\begin{Theorem}\label{T1}
Let $I\subset R$ be a polymatroidal ideal in degree $2$. Then $\astab(I)=\dstab(I)$. In particular, if $I$ is a matroidal ideal in degree $2$, then $\astab(I)=\dstab(I)$.
\end{Theorem}

\begin{proof}
We can assume that $I$ is a full-supported polymatroidal ideal in degree $2$ and $\gcd(I)=1$. Now,
we consider two cases.\\
{\it Case 1}: Let ${\mm}\notin\Ass^{\infty}(I)$.\\
By Remark \ref{R1}(e), for all $t\geq 1$, $\Ass(I^t)=\cup_{i=1}^n\Ass(I[i]^t)$. Since $\Ass(I[i]^t)=\Ass(I[i])$ for $i=1,...,n$, it therefore follows $\astab(I)=1$. By using Lemma \ref{L2} we have $\depth R/I=1$. Since ${\mm}\notin\Ass^{\infty}(I)$ it follows $\depth R/I^t>0$ for all $t\geq 1$. Since $\depth R/I^t\leq\depth R/I=1$ for all $t\geq 1$, we have $\depth R/I^t=1$ for all $t\geq 1$. Therefore $\dstab(I)=1$ and so the result follows in this case.\\
{\it Case 2}: Let ${\mm}\in\Ass^{\infty}(I)$.\\
Let $t$ be the largest integer number such that ${\mm}\notin\Ass(I^t)$. By Remark \ref{R1}(e), $\Ass(I^{s+1})\setminus\{\mm\}=\cup_{i=1}^n\Ass(I[i]^{s+1})$ for all $s\geq t$.
 Since $\Ass(I[i]^{s+1})=\Ass(I[i]^s)$ for $i=1,...,n$ and for all $s\geq 1$ , we have $\Ass(I^{s+1})=\Ass(I^t)\cup\{\mm\}$ for all $s\geq t$. Hence $\astab(I)=\dstab(I)=t+1$.
\end{proof}

\begin{Corollary}\label{C1}
Let $I\subset R$ be a matroidal ideal in degree $2$. Then $\astab(I)=\dstab(I)\leq 2$.
\end{Corollary}

\begin{proof}
If ${\mm}\notin\Ass^{\infty}(I)$, then by the proof of Theorem \ref{T1} $\astab(I)=\dstab(I)=1$. \\
Let ${\mm}\in\Ass^{\infty}(I)$. Since $I$ is a edge ideal of a connected graph $G$, it follows that $G$ has a cycle of length odd. Therefore by \cite[Lemma 2.3]{C}
$G$ has a cycle of length 3 and so by \cite[Theorem 3.1]{HH3} $\depth R/I^2=0$. Thus by Theorem \ref{T1} we have $\astab(I)=\dstab(I)=2$.
\end{proof}

We recall the following definition from \cite{BH} (or see \cite{HH4}).
\begin{Definition}
Let $I$ be a monomial ideal. We say that $I$ satisfies the strong exchange property if $I$ is generated in a single degree, and for all $u,v\in G(I)$ and for all $i,j$ with $\deg_{x_i}(u)>\deg_{x_i}(v)$ and $\deg_{x_j}(u)<\deg_{x_j}(v)$, one has $x_j(u/x_{i})\in G(I)$.
\end{Definition}

Note that if $n=3$, then every polymatroidal ideal $I$ with $\gcd(I)=1$, satisfying the strong exchange property (see \cite[Propsition 2.7]{BH}).

One of the most distinguished polymatroidal ideals is the ideal of Veronese type. Consider the fixed positive integers $d$ and $1\leq a_1\leq ...\leq a_n\leq d$. The ideal of Veronese type of $R$ indexed by $d$ and $(a_1,...,a_n)$ is the ideal $I_{(d;a_1,...,a_n)}$ which is generated by those monomials $u=x_1^{i_1}...x_n^{i_n}$ of $R$ of degree $d$ with $i_j\leq a_j$ for each $1\leq j\leq n$.

\begin{Proposition}\label{P4}
Let $I\subset R$ be a polymatroidal ideal satisfying the strong exchange property. Then $\astab(I)=\dstab(I)$.
\end{Proposition}

\begin{proof}
By Lemma \ref{L3}, we can assume that $\gcd(I)=1$ and $I$ is full-supported. Therefore by using \cite[Theorem 1.1]{HHV} $I$ is a polymatroidal ideal of Veronese type.
Now by using \cite[Corollary 5.7]{HRV} we have $\astab(I)=\dstab(I)$, as required.
\end{proof}

\begin{Lemma}\label{L4}
Let $I\subset R$ be a full-supported matroidal ideal of degree $n-1$, where $n\geq 2$ and $\gcd(I)=1$. Then $I=I_{n-1;1,1,...,1}$ as a Veronese type ideal.
\end{Lemma}

\begin{proof}
 For $n=2,3$, a matroidal ideal $I$ satisfies in the strong exchange property and so $I$ is a square-free Veronese ideal and the result follows in this case.  Now, let $n\geq 4$. Then $\depth R/I=n-2$ and so $\pd(R/I)=2$. Therefore $\height(I)=\pd(R/I)$ and so $I$ is a Cohen-Macaulay ideal. By using \cite[Theorem 4.2]{HH2} $I=I_{n-1;1,...,1}$, as desired.
\end{proof}

\begin{Theorem}\label{T2}
Let $n=4$ and $I\subset R$ be a matroidal ideal of degree $d$. Then $\astab(I)=\dstab(I)$.
\end{Theorem}

\begin{proof}
We can assume that $\gcd(I)=1$. If $d\neq 2,3$, then the result is clear. If $d=2$, then by Theorem \ref{T1} the result follows. For $d=3$, the result holds by Lemma \ref{L4}.
\end{proof}

\begin{Remark}\label{R2}
Let $n=4$ and $I\subset R$ be a matroidal ideal of degree $d$. Then $\astab(I),\dstab(I)\leq d$.
\end{Remark}

\begin{proof}
If $d=1,4$, then the result is clear. If $d=2$, then by Corollary \ref{C1}, $\astab(I)=\dstab(I)\leq 2$, and for $d=3$ the result follows by \cite[Corollary 5.7]{HRV}. This completes the proof.
\end{proof}

In view of the above remark one may ask whether $\astab(I),\dstab(I)\leq d$ for all matriod ideals $I$.

\begin{Theorem}\label{T3}
Let $n=4$ and $I\subset R$ be a polymatroidal ideal  of degree $d$ and let ${\mm}\notin\Ass^{\infty}(I)$. Then $\astab(I)=\dstab(I)$.
\end{Theorem}

\begin{proof}
We can assume that $\gcd(I)=1$. Therefore $\height(I)\geq 2$. Since ${\mm}\notin\Ass^{\infty}(I)$, we have $\depth R/I^t\geq 1$ for all $t\geq 1$. Thus $2\leq\height(I)\leq\pd(R/I)\leq 3$. If $\height(I)=\pd(R/I)$, then by \cite[Theorem 4.2]{HH2} and \cite[Corollary 5.7]{HRV} we have $\astab(I)=\dstab(I)$. Now, let $\pd(R/I)=3$ and $\height(I)=2$. Hence $\depth R/I^t=\depth R/I=1$ for all $t\geq 1$ and so $\dstab(I)=1$. By Proposition \ref{P2} there exists an integer $i$ such that $\astab(I)=\astab(I[i])$, where  $i=1,..,4$.
By \cite[Theorem 1.2]{HM}, $\astab(I[i])=\dstab(I[i])$ and so $\astab(I)=\astab(I[i])=\dstab(I[i])$. For simplicity let $\astab(I)=\astab(I[1])=\dstab(I[1])$.
Since $\gcd(I)=1$, we easily obtain that $\gcd(I[i])=1$ for all $i=1,...,4$ and so $\height(I[i])=2$; since $\Ass(I[i])\subseteq\Ass(I)$. Therefore  $\depth R/I[1]=1$ or $2$. If $\depth R/I[1]=1$, then $\depth R/I^t[1]=1$ for all $t$. Hence $\dstab(I[1])=1$ and so $\astab(I)=1$. Thus the result follows in this case. Now let $\depth R/I[1]=2$. Then $\pd(R/I[1])=2=\height(I[1])$. Hence $I[1]$ is a Cohen-Macaulay
polymatroidal ideal and by \cite[Theorem 4.2]{HH2} $I[1]$ is a square-free Veronese ideal.
Since ${\mm}\notin\Ass^{\infty}(I)$, it follows $\Ass(\Hom_R(R/(x_1^{t_1},...,x_4^{t_4}), R/I))=\emptyset$ for all non-negative integer $t_1,...,t_4$. Thus
$I=\cap_{i=1}^4I[i]$. If for all $i=1,...,4$, $\depth R/I[i]=2$, then it follows that $I$  is a matriod ideal and so by Theorem \ref{T2} the result follows in this case. If $I[1]$ is a square-free Veronese ideal in the polynomial ring with two variables, then again the result follows.
Let $I[1]$ be a square-free Veronese ideal in the polynomial ring $K[x_2,x_3,x_4]$ and let $\depth R/I[2]=1$. Since for  all $i=1,...,4$, $I[i]$ is polymatrodial ideal in the polynomial ring with at most three variables and $\gcd(I[i])=1$, it follows that $I[i]$ is a Veronese-type ideal. We can assume that $I[2]$ is a Veronese-type ideal in $K[x_1,x_3,x_4]$. Let $a_i=\max\{\deg_{x_i}(u): u\in G(I)\}$ for $i=1,...,4$.
Since $I[1][2]=I[2][1]$, we have $a_3=1=a_4$. If $I[3]$ is a Veronese-type ideal in the polynomial ring $K[x_1,x_2,x_4]$, then from $I[1][3]=I[3][1]$ we have $a_2=1$ and also from $I[2][3]=I[3][2]$ we have $a_1=1$. Therefore $I$ is a matroid ideal and the result follows.
We can assume that $I[3]$ is a square-free Veronese ideal in the polynomial ring $K[x_1,x_2]$ and so $I[3]=(x_1,x_2)$. From $I[2][3]=I[3][2]$, we have
$K[x_1,x_2]=(x_1,x_4)$ and this is a contradiction. Thus from all of the above cases we have $\astab(I)=\dstab(I)$, as required.

\end{proof}

\begin{Lemma}
 $I\subset R$ be a polymatroidal ideal and ${\mm}\in\Ass^{\infty}(I)$. Then $\dstab(I)\leq\astab(I)$. In particular, if $\dstab(I)=n-1$ then $\astab(I)=n-1$ or if $\astab(I)=1$, then $\dstab(I)=1$.
\end{Lemma}

\begin{proof}
Since ${\mm}\in\Ass^{\infty}(I)$ by \cite[Corollary 1.6]{HQ} we have $\ell(I)=\dim R$. Therefore $\dstab(I)=\min\{k: \depth R/I^k=0\}=\min\{k: {\mm}\in\Ass(I^k)\}$ and so $\dstab(I)\leq\astab(I)$.
\end{proof}

Herzog and Qureshi \cite{HQ} conjectured that $\astab(I)=\dstab(I)$ for all polymatroidal ideal $I$.
The following examples disprove Herzog and Qureshi's conjecture, in general.
\begin{Example}\label{E}
Let $n=4$ and \[
I=(x_1x_2x_3,x_2^2x_3,x_2x_3^2,x_1x_2x_4,x_2^2x_4,x_2x_4^2,x_1x_3x_4,x_3^2x_4,x_3x_4^2,x_2x_3x_4)
.\] Then $\dstab(I)=1$ and $\astab(I)=2$.
\end{Example}

\begin{proof}
It is clear $I$ is a polymatroidal ideal. By using \cite[Theorem 4.1]{HQ}, we have $\dstab(I)=1$ and $\astab(I)=2$, as desired.
\end{proof}

The following example is a generalization of Example \ref{E}.
\begin{Example} Let $n\geq 4$ and
$I=(x_1x_3...x_{n},x_3^2x_4...x_{n},x_3x_4^2...x_{n},...,x_3x_4...x_{n}^2,\\ x_1x_2x_4...x_n,x_2^2x_4...x_n,...,x_2x_4...x_n^2,...,x_1x_2...x_{n-1},
x_2^2x_3...x_{n-1},...,x_2x_3...x_{n-1}^2,x_2x_3...x_n).$ Then $I$ is a polymatroidal ideal with $\dstab(I)=1$ and $\astab(I)=n-2$.
\end{Example}

\begin{proof}
Since $I=I_{n-1;1,1,..,1}+I_{n-1;0,2,2,...,2}$, we easily conclude that $I$ is a polymatroidal ideal. Since $I:x_3x_4...x_n=\frak{m}$, we have $\dstab(I)=1$. For all $i=1,2,...,n,$ $I[i]$ is a square-free Veronese ideal. Thus by \cite[Corollary 5.7]{HRV} $\astab(I[1])=n-2$ and $\astab(I[i])=n-3$ for all $i=2,3,...,n$. Therefore by using Proposition \ref{P2} we have $\astab(I)=n-2$.
\end{proof}
\begin{Lemma}\label{L5}
Let $n=5$ and $I\subset R$ be a matroidal ideal with $\gcd(I)=1$ and $d=3$. Then $\astab(I)\neq 1$.
\end{Lemma}

\begin{proof}
By contrary, we assume that $\astab(I)=1$. By Proposition \ref{P2}, for all $i=1,...,5$, $\astab(I)=\astab(I[i])=\dstab(I[i])=1$.
Since $\gcd(I[i])=1$ and $\dstab(I[i])=1$, we conclude that, for $i=1,2,...,5$, $I[i]$ is a transversal matroidal ideal in polynomial ring with four variables and $\height(\frak{p})=2$ for all $\frak{p}\in\Ass(I)$. Let $I[1]={\frak{p}_1}{\frak{p}_2}, I[2]=\frak{p}_3\frak{p}_4, I[3]=\frak{p}_5\frak{p}_6, I[4]=\frak{p}_7\frak{p}_8$ and $I[5]=\frak{p}_9\frak{p}_{10}$. Therefore we can assume that $\frak{p}_1=\frak{p}_3=\frak{p}_5=(x_4,x_5)$ and $\frak{p}_2=\frak{p}_7=\frak{p}_9=(x_2,x_3)$. Thus $I[2][4]=(x_1,x_3)$ and $I[4][2]=(x_1,x_5)$ and this is a contradiction. Hence $\astab(I)\neq 1$, as required.
\end{proof}

\begin{Theorem}\label{T4}
Let $n=5$ and $I\subset R$ be a matroidal ideal of degree $d$. Then $\astab(I)=\dstab(I)$.
\end{Theorem}

\begin{proof}
We can assume that $\gcd(I)=1$ and $I$ is a full-supported matroidal ideal. If $d\neq 2,3,4$, then there is nothing to prove. For $d=4$ and $d=2$ the result follows by Lemma \ref{L4} and Theorem \ref{T1}, respectively. Thus we can assume that $d=3$. Let ${\mm}\in\Ass^{\infty}(I)$. Then, by Remark \ref{R1}(e) and Corollary \ref{C1}, $\Ass^{\infty}(I)=\Ass(I)\cup\{\mm\}$ or
$\Ass^{\infty}(I)=\Ass(I^2)\cup\{\mm\}$. Hence $\astab(I)=\dstab(I)$.
Now, let ${\mm}\notin\Ass^{\infty}(I)$. Then, by Proposition \ref{P2}, $\astab(I)=\astab(I[i])=\dstab(I[i])$ for some integer $i=1,2,...,5$. Let $I[i]$ be a full-supported ideal in $k[x_1,...,\widehat{x_i},...,x_5]$. If $(x_1,...,\widehat{x_i},..,x_5)\notin\Ass^{\infty}(I)$, then by Corollary \ref{C1}, $\astab(I[i])=1$ and this is a contradiction with Lemma \ref{L5}. Therefore $(x_1,...,\widehat{x_i},..,x_5)\in\Ass^{\infty}(I)$ and so by \cite[Corollary 1.6]{HQ} we have $\ell(I[i])=4$.
 It is well known that $\ell(I[i])\leq \ell(I)$ and so $\ell(I[i])=\ell(I)=4$. Since $I[i]$ is a matroidal ideal of degree 2 and $(x_1,...,\widehat{x_i},..,x_5)\in\Ass^{\infty}(I)$, by Corollary \ref{C1} we have $\dstab(I[i])=2$. Therefore by using \cite[Theorem 3.1]{CHHKTT}
it follows $\dstab(I)=\dstab(I[i])$ and so $\astab(I)=\dstab(I)$. Let $I[i]$ be a matroidal ideal in polynomial ring with three variables. Since ${\mm}\notin\Ass^{\infty}(I)$ and $\dstab(I)\neq 1$, we have $s=2$; where $s$ is the number of connected components of $\Gamma_I$. Therefore we have $I\subseteq{\fp}_1\cap{\fp}_2$ and so we can assume that
${\fp}_1=(x_1,x_2)$ and ${\fp}_2=(x_3,x_4,x_5)$. Thus $(I:x_3x_4)=(I:x_3x_5)=(I:x_4x_5)={\fp}_1$. Since $\Gamma_I$ is disconnected with connected components $\Gamma_1$ and $\Gamma_2$ in which ${\fp}_1$ and ${\fp}_2$ are the monomial prime ideals generated by the set of vertices $\Gamma_1$ and $\Gamma_2$, respectively, we have $x_1x_2x_i\notin I$ for $i=3,4,5$ and $x_3x_4x_5\notin I$. It therefore follows $I=(x_1x_3x_4,x_1x_4x_5,x_1x_3x_5,x_2x_3x_4,x_2x_4x_5,x_2x_3x_5)$. Thus $\dstab(I)=\astab(I)$, as required.
\end{proof}

\subsection*{Acknowledgements}
We would like to thank deeply grateful to the referee for the careful reading
of the manuscript and the helpful suggestions.

%%%%%%%%%%%%%%%%%%%%%%%%%%%%%%%%%%%%%%%%%%%%%%%%%%%%%%%%%%%%%%%%%%%%%%%%%%%%%


\begin{thebibliography}{99}
\bibitem{AB} J. Abbott and A. M. Bigatti, {\it a C++ library for doing Computations in Commutative Algebra}, Available at
    http://cocoa.dima.unige.it/cocoalib.
\bibitem{BH} S. Bandari and J. Herzog, {\it Monomial localizations and polymatroidal ideals}, Eur. J. Comb., {\bf 34}(2013), 752-763.
\bibitem{BHR} S. Bayati, J. Herzog and G. Rinaldo, {\it A routine to compute the stable set of associated prime ideals of amonomial ideal}, Available at
    http://www.giancarlorinaldo.it/stableset.html.
\bibitem{B1} M. Brodmann, {\it Asymptotic stability of $\Ass(M/{I^nM})$}, Proc. Am. Math. Soc., {\bf 74}(1979), 16-18.
\bibitem{B} M. Brodmann, {\it The asymptotic nature of the analytic spread}, Math. Proc. Cambridge Philos. Soc., {\bf 86}(1979), 35-39.
\bibitem{BH1} W. Bruns and J. Herzog, {\it Cohen-Macaulay rings}, Cambridge University Press, Cambridge, UK, (1998).
\bibitem{CHHKTT} G. Caviglia, H. T. Ha, J. Herzog, M. Kummini, N. Terai and N. V. Trung, {\it Depth and regularity modulo and principal ideal},
J. Algebr Comb (2018). https://doi.org/10.1007/s10801-018-0811-9.
\bibitem{GS}D. R. Grayson and M. E. Stillman, {\it Macaulay 2, a software system for research in algebraic geometry}, Available at
    {http://www.math.uiuc.edu/Macaulay2/}.
\bibitem{C} H. J. Chiang-Hsieh, {\it Some arithmetic properties of matroidal ideals}, Comm. Algebra, {\bf 38}(2010), 944-952.
\bibitem{CH} A. Conca and J. Herzog, {\it Castelnuovo-Mumford regularity of products of ideals}, Collect. Math., {\bf 54}(2003), 137-152.
\bibitem{EH} D. Eisenbud and C. Huneke, {\it Cohen-Macaulay Rees algebra and their specialization}, J. Algebra, {\bf 81}(1983), 202-224.
\bibitem{HH3} J. Herzog and T. Hibi,{\it Bounding the socles of powers of squarefree monomial ideals}, Commutative algebra and noncommutative algebraic geometry, {\bf 68}(2015), 223-229.
\bibitem{HH1} J. Herzog and T. Hibi, {\it Monomial ideals}, Grad.Texts Math., vol.260, Springer-Verlag London, Ltd., London, (2011).
\bibitem{HH2} J. Herzog and T. Hibi, {\it Cohen-Macaulay polymatroidal ideals}, Eur. J. Comb., {\bf 27}(2006), 513-517.
\bibitem{HH} J. Herzog and T. Hibi, {\it The depth of powers of an ideal}, J. Algebra, {\bf 291}(2005), 534-550.
\bibitem{HH4} J. Herzog and T. Hibi, {\it Discrete polymatroids}, J. Algebr Comb., {\bf 16}(2002), 239-268.
\bibitem{HHV} J. Herzog, T. Hibi and M. Vladoiu, {\it Ideals of fiber type and polymatroids}, Osaka J. Math., {\bf 42}(2005), 807-829.
\bibitem{HM} J. Herzog and A. Mafi, {\it Stability properties of powers of ideals in regular local rings of small dimension}, Pacific J. Math., {\bf 295}(2018), 31-41.
\bibitem{HQ} J. Herzog and A. Qureshi, {\it Persistence and stability properties of powers of ideals}, J. Pure and Applied Algebra, {\bf 219}(2015),
    530-542.
\bibitem{HRV} J. Herzog, A. Rauf and M. Vladoiu, {\it The stable set of associated prime ideals of a polymatroidal ideal}, J. Algebr Comb., {\bf 37}(2013), 289-312.
\bibitem{HT} J. Herzog and Y. Takayama, {\it Resolutions by mapping cones}, Homology Homotopy Appl., {\bf 4}(2002), 277-294.
\bibitem{HV} J. Herzog and M. Vladoiu, {\it Squarefree monomial ideals with constant depth function}, J. Pure and Applied Algebra, {\bf 217}(2013),
    1764-1772.
\bibitem{HS} C. Huneke  and I. Swanson. {\it Integral closure of ideals, rings, and modules},  vol. {\bf 13},  Cambridge University Press, 2006.
\bibitem{KR} D. Katz and L. J. Ratliff, JR., {\it On the prime divisors of IJ when I is integrally closed}, Arch. Math., {\bf 51}(1988), 520-522.
\bibitem{T} T. N. Trung, {\it Stability of associated primes of integral closures of monomial ideals}, J. Comb. Theory Ser. A {\bf 116}(2009), 44-54.
\bibitem{V} R. H. Villarreal, {\it Monomial Algebras}, Marcel Dekker, Inc, New York, (2001).
\end{thebibliography}
\end{document}